\documentclass{article}
\usepackage{amsthm}
\usepackage{commath}
\usepackage{amsmath}
\usepackage{amssymb}
\usepackage{mathtools}
\usepackage{amsfonts}
\usepackage{hyperref}
\usepackage{array}
\usepackage{xcolor}
\usepackage{lipsum}
\usepackage{bibentry}

\newtheorem{theorem}{Theorem}[section]
\newtheorem{corollary}{Corollary}[theorem]
\newtheorem{lemma}[theorem]{Lemma}
\newtheorem{conjecture}[theorem]{Conjecture}

\definecolor{lightgray}{gray}{0.8}
\newcolumntype{L}{>{\raggedleft}p{0.14\textwidth}}
\newcolumntype{R}{p{0.8\textwidth}}

\begin{document}

\title{On a Double Series Representation of the Natural Logarithm, the Asymptotic Behavior of Hölder Means, and an Elementary Estimate for the Prime Counting Function}

\author{Sinan Deveci}
\date{}
\maketitle
\thispagestyle{empty}

\begin{abstract}
We present many novel results in number theory, including a double series formula for the natural logarithm and a proof concerning the Hölder mean based on the functional equation for the Riemann zeta function. We find a harmonic mean analogue of Chebyshev's inequality for the prime counting function involving the Euler-Mascheroni constant. Furthermore, we define a function taking the Hölder mean of all positive integers up to a given number and investigate its asymptotic behavior, finding two different patterns which are separated by the harmonic mean. Additionally, we discuss the behavior of said function at zero and discover a formula involving the Riemann zeta function, whose continuity we prove with Riemann's functional equation. Inspired by the alternating harmonic series, we find a double series formula for the natural logarithm, resulting in identities involving the Riemann zeta function, binomial coefficients, and logarithms.

\end{abstract}

\newpage
\pagenumbering{roman}

\tableofcontents

\newpage
\pagenumbering{arabic}

\section{Introduction}
We generalize the equivalence between the alternating harmonic series and the natural logarithm of 2 and find a double series formula for the natural logarithm of all positive integers $\geq 2$ in Section \ref{logs}. We then proceed by raising the individual terms to positive integer powers in two different ways and find formulae involving again the Riemann zeta function and powers of $\pi$  as well as central entries of Pascal's triangle and thus binomial coefficients.
\\
\indent In Section \ref{pyth} we give an introduction to the Hölder mean, define a function that takes the $m$-mean of all positive integers less than or equal to a given number $x$, and investigate the dominant term of its asymptotic behavior. Approaching the value -1 for $m$, we find a sequence that is related to the discriminants of the Chebyshev $C$-Polynomials, as given by the OEIS $\cite{chebypoly}$. 
We extend the domain of this function to include 0 and find a relation to the Bernoulli numbers. We find a proof for the discontinuity at $m=0$ involving the functional equation for the Riemann zeta function.
\\
\indent Inspired by an inequality of Chebyshev, we find a bound for the prime counting function dependent on the harmonic mean, as shown in Section \ref{primes}. This allows us to estimate the number of primes in any subset of $\mathbb{N}$ using only elementary operations.
\\
\indent Analogous to how introducing factorials to the denominators of the geometric series gives the Taylor series expansion of the exponential functions, we introduce factorials to the denominators of arithmetic and harmonic series in Section \ref{pythsrs}, and thereby rediscover the Bell numbers.

\newpage

\section{Prime Counting Function}\label{primes}
As one cannot talk about (analytic) number theory without mentioning the prime number theorem, we shall devote the first section to the building blocks of mathematics, the primes. The most famous result about their distribution, the prime number theorem (PNT), establishes the following asymptotic equivalence:

\begin{equation}
	\label{PNT}
	\pi(x)\sim\frac{x}{\log{x}},
\end{equation}
i.e. $\displaystyle\lim_{x\rightarrow \infty} \frac{\pi(x)}{\frac{x}{\log{x}}}=1,$
where $\pi(x)$ is the number of primes $p\leq x$ and $\log$ is the natural logarithm. An important number in number theory is the Euler-Mascheroni constant, which is defined by
\begin{equation}
\label{eulergamma}
\gamma\coloneqq\lim_{n\rightarrow\infty} \Bigg(\sum_{k=1}^{n}\frac{1}{k} - \log{n}\Bigg)=\lim_{n\rightarrow\infty}\big(H_n-\log{n}\big)\approx 0.57721,
\end{equation}
where $H_n$ is the $n$th harmonic number. Since $H_n$ and $\log{n}$ are so close, the idea of combining equations \eqref{PNT} and \eqref{eulergamma} comes naturally. However, since $H_n$ is still substantially larger than $\log{n}$ for small values of $n$, modifications of $H_n$ have been used. Locker-Ernst \cite{locker-ernst} has used $H_n-\frac{3}{2}$, while Hua \cite{hua} used $H_n-1$. Although these will most probably give better approximations to $\pi(x)$, we will nevertheless investigate the case of using the ``original'' harmonic numbers. Thus we begin by claiming the following:
\begin{theorem}
\label{thm1}
\[\pi(x) \sim \frac{x}{H_x}=HM(x),\]
where $HM(x)$ denotes the harmonic mean of all positive integers $\leq x$.
\end{theorem}

\begin{proof}
\hypertarget{pf2.1}{}
We have to prove that
\[\lim_{x\rightarrow\infty} \frac{\pi(x)}{x}\cdot H_x=1.\]
From
\[\frac{\pi(x)}{x}\cdot H_x-\frac{\pi(x)}{x}\cdot \log{x}=\frac{\pi(x)}{x}\cdot (H_x-\log{x})\]
and
\begin{flalign*}
\lim_{x\rightarrow\infty} \frac{\pi(x)}{x}\cdot (H_x-\log{x})&=\lim_{x\rightarrow\infty} \frac{\pi(x)}{x}\cdot \lim_{x\rightarrow\infty} (H_x-\log{x}) \\
&=0\cdot\gamma=0
\end{flalign*}
we get
\[\lim_{x\rightarrow\infty}\bigg(\frac{\pi(x)}{x}\cdot H_x - \frac{\pi(x)}{x}\cdot\log{x}\bigg)=0.\]
Applying the PNT we end up with
\[\lim_{x\rightarrow\infty}\frac{\pi(x)}{x}\cdot H_x = \lim_{x\rightarrow\infty} \frac{\pi(x)}{x}\cdot\log{x}=1.\]
\end{proof}

Thus we have established that the number of primes less than or equal to a given magnitude is asymptotically equal to the harmonic mean of all positive integers up to that magnitude. And that to a certain extent again hints at the close connection that a very exclusive set of numbers has to all others, that the primes - although full of mystery and seemingly with no pattern whatsoever - are still deeply embedded in the nature of the integers.
From this discovery, given that some properties of the harmonic mean are known, two statements immediately follow.
Note that even though they follow directly from the PNT, they also follow from \thmref{thm1}. A proof without using the PNT exists but lies beyond the scope of this paper.
\begin{corollary}
There are infinitely many primes.
\end{corollary}
\begin{proof}
The harmonic mean $HM(x)$ increases without bound as $x$ increases without bound. Since $\pi(x)$ is asymptotically equal to $HM(x)$, it must also increase without bound, therefore there are infinitely many primes.
\end{proof}
\begin{corollary}
The natural density of primes approaches zero, i.e. the average prime gaps get larger and larger.
\end{corollary}
\begin{proof}
The harmonic mean of a collection of values tends to stay close to the minimum value. In fact, $\frac{HM(x)}{x}$ approaches 0 as $x$ approaches $\infty$. From these well known facts about the harmonic mean we conclude that since $\pi(x)$ is asymptotically equivalent to $HM(x)$, the quotient $\frac{\pi(x)}{x}$ must also approach 0 as $x$ approaches $\infty$.
\end{proof}

\noindent After we have discovered how the prime counting function $\pi(x)$ relates to the harmonic mean $HM(x)$ for large values of $x,$ we now investigate how these two quantities relate for all integers $x\geq 2$ by means of bounding coefficients.
\\
\begin{theorem}
\label{thm11}
The inequality
\[\frac{1}{6}<\frac{\pi(x)}{\frac{x}{H_x}}<6+\frac{2}{3}\]
holds for any integer $x\geq 2$.
\end{theorem}

\begin{proof}
\hypertarget{pf2.2}{}
Chebyshev's inequality \cite{rassias} states that
\[\frac{1}{6}\cdot\frac{x}{\pi(x)} < \log{x} < 6\cdot\frac{x}{\pi(x)}.\]
Moreover, since 1 is the maximum value of $H_x -\log{x}$ and by \eqref{eulergamma} we have
\[0<\gamma\leq H_x-\log{x}\leq 1\]
for any $x\geq 2.$ This gives
\begin{flalign*}
\frac{1}{6}\cdot\frac{x}{\pi(x)}&\leq\gamma+\frac{1}{6}\cdot\frac{x}{\pi(x)}\leq\gamma+\log{x}\leq H_x \\
&\leq 1+\log{x}<1+6\cdot\frac{x}{\pi(x)}.
\end{flalign*}
Dividing by $\frac{x}{\pi(x)}$ provides
\[\frac{1}{6}\leq H_x \cdot\frac{\pi(x)}{x}\leq\frac{\pi(x)}{x}+6\leq\frac{\pi(3)}{3}+6=6+\frac{2}{3}.\]
\end{proof}

\noindent Although these bounds are slightly weaker than the ones from Chebyshev's inequality, this approximation has the advantage that it uses harmonic means instead of logarithms, and is thus computable using only elementary operations. As we have made extensive use of the harmonic mean, we shall now investigate means in general.
\newpage

\section{Hölder Means}\label{pyth}

Since antiquity the question has been around of how to find the mean value for a given collection of values. There is a multitude of means or averages, the three most well-known being the three Pythagorean means, consisting of the\\\\
Arithmetic Mean
\[AM(x_1,...,x_n)=\frac{x_1+...+x_n}{n}=\frac{1}{n}\sum \limits_{i=1}^{n} x_i,\]
the Geometric Mean
\[GM(x_1,...,x_n)=\sqrt[n]{x_1\cdot\cdot\cdot\cdot\cdot x_n}=\sqrt[n]{\prod\limits_{i=1}^{n} x_i},\]
and the Harmonic Mean
\[HM(x_1,...,x_n)=\frac{n}{\frac{1}{x_1}+...+\frac{1}{x_n}}=\frac{n}{\sum \limits_{i=1}^{n} \frac{1}{x_i}}.\]
Generalization yields the power mean \cite{powermean} or Hölder mean, after the German mathematician Otto Ludwig Hölder (1859-1937) \cite{holder}
\[
	M_m(x_1,...,x_n)=\sqrt[m]{\frac{1}{n}\sum\limits_{i=1}^{n} x_i^{m}} \hspace*{5mm}\text{ for any } m \in \mathbb{R}\setminus \{0\},
\]
where $x_i$ are positive real numbers. Considering the special cases $m=1$ and $m=-1$ we obtain the arithmetic and harmonic mean, respectively.
\begin{theorem}
The geometric mean is obtained by taking the limit of the power mean as $m$ approaches $0.$
\end{theorem}

\begin{proof}
Since we are trying to relate a product to a sum, it seems natural to make use of the properties of the logarithm and the exponential function. \\
By definition,
\begin{flalign*}
M_m(x_1,...,x_n)&=\sqrt[m]{\frac{1}{n}\sum\limits_{i=1}^{n} x_i^{m}}
=\Bigg({\frac{1}{n}\sum\limits_{i=1}^{n} x_i^{m}}\Bigg)^{\frac{1}{m}}\\
&=\exp{\log{\Bigg({\frac{1}{n}\sum\limits_{i=1}^{n} x_i^{m}}\Bigg)^{\frac{1}{m}}}}
=\exp{\frac{\log{\Bigg({\frac{1}{n}\sum\limits_{i=1}^{n} x_i^{m}}\Bigg)}}{m}}.
\end{flalign*}
Now we take the limit as $m$ approaches $0$:
\[
	\lim_{m \rightarrow 0}M_m(x_1,...,x_n)=\lim_{m \rightarrow 0}\exp{\frac{\log{\Bigg({\frac{1}{n}\sum\limits_{i=1}^{n} x_i^{m}}\Bigg)}}{m}}.
\]
Since $\exp{}$ is continuous, we can take it outside of the limit:
\begin{equation}
	\label{LDrei}
	\lim_{m \rightarrow 0}M_m(x_1,...,x_n)=\exp{\lim_{m \rightarrow 0}\frac{\log{\Bigg({\frac{1}{n}\sum\limits_{i=1}^{n} x_i^{m}}\Bigg)}}{m}}.
\end{equation}
Now we have the exponential of a limit on the RHS. Let us define
\[
	L\colon=\lim_{m \rightarrow 0}\frac{\log{\Bigg({\frac{1}{n}\sum\limits_{i=1}^{n} x_i^{m}}\Bigg)}}{m}.
\]
Since $\log{}$ is differentiable on $(0,\infty)$, and we have a $\frac{0}{0}$ situation, we can apply L'Hôpital's rule to obtain
\[
	L=\lim_{m \rightarrow 0}\frac{\frac{d}{dm}\Bigg(\log{\Bigg({\frac{1}{n}\sum\limits_{i=1}^{n} x_i^{m}}\Bigg)}\Bigg)}{\frac{d}{dm}(m)}.
\]
Using the chain rule gives
\begin{flalign*}
L&=\lim_{m \rightarrow 0}\frac{\frac{d}{dm}\Bigg(\frac{1}{n}\sum\limits_{i=1}^{n} x_i^{m}\Bigg)}{1\cdot\frac{1}{n}\sum\limits_{j=1}^{n} x_j^{m}}=\lim_{m \rightarrow 0}\frac{\frac{1}{n}\sum\limits_{i=1}^{n} \big(x_i^{m}\cdot\log{x_i}\big)}{1\cdot\frac{1}{n}\sum\limits_{j=1}^{n} x_j^{m}} \\
&=\lim_{m \rightarrow 0}\frac{\sum\limits_{i=1}^{n} \big(x_i^{m}\cdot\log{x_i}\big)}{\sum\limits_{j=1}^{n} x_j^{m}}.
\end{flalign*}
Since the quotient of a sum is equal to the sum of quotients, where the denominator $\neq 0$, we can rewrite $L$ as 
\[
	L=\lim_{m \rightarrow 0}\sum\limits_{i=1}^{n}\frac{ (x_i^{m}\cdot\log{x_i})}{\sum\limits_{j=1}^{n} x_j^{m}}.
\]
Expanding every term of the sum by $\frac{1}{x_i^{m}}$ gives
\[
	L=\lim_{m \rightarrow 0}\sum\limits_{i=1}^{n}\frac{\log{x_i}}{\frac{1}{x_i^{m}}\sum\limits_{j=1}^{n} x_j^{m}}
	=\lim_{m \rightarrow 0}\sum\limits_{i=1}^{n}\frac{\log{x_i}}{\sum\limits_{j=1}^{n}\frac{x_j^{m}}{x_i^{m}}}.
\]
Since only the denominator is dependent on $m$, we can equivalently write
\[
	L=\sum\limits_{i=1}^{n}\frac{\log{x_i}}{\displaystyle\lim_{m \rightarrow 0}\sum\limits_{j=1}^{n}\frac{x_j^{m}}{x_i^{m}}}.
\]
We note that all $x_j, x_i$ are positive and $x^{0}=1$ for $x\neq 0$, hence $\displaystyle\lim_{m\rightarrow 0}\frac{x_j^{m}}{x_i^{m}}=1$ for any $1\leq i, j\leq n$ and thus
\[\lim_{m \rightarrow 0}\sum\limits_{j=1}^{n}\frac{x_j^{m}}{x_i^{m}}=n,\]
and therefore
\[
	L=\sum\limits_{i=1}^{n}\frac{\log{x_i}}{n}
	=\sum\limits_{i=1}^{n}\log\Big({x_i^{\frac{1}{n}}\Big)}.
\]
Since $\exp{}$ is continuous, we can resubstitute this expression for $L$ into \eqref{LDrei} to get
\[
	\lim_{m \rightarrow 0}M_m(x_1,...,x_n)=\exp{\lim_{m \rightarrow 0}\frac{\log{\Bigg({\frac{1}{n}\sum\limits_{i=1}^{n} x_i^{m}}\Bigg)}}{m}}
	=\exp{L}
	=\exp{\sum\limits_{i=1}^{n}\log\Big({x_i^{\frac{1}{n}}\Big)}}.
\]
By the very intrinsic property of the exponential function to convert multiplication into addition, we can conclude that
\[
	\lim_{m \rightarrow 0}M_m(x_1,...,x_n)=\prod\limits_{i=1}^{n} x_i^{\frac{1}{n}}.\]
\end{proof}
Our proof can be found in \cite{powermean}. Let us now define the function
\[
	M_m(x)\coloneqq M_m(1,...,x)=M_m(x_1,...,x_n),
\]
where $x_i=i\in\mathbb{N}$ for $1\leq i\leq x\in\mathbb{N}$.
We now investigate how $M_m(x)$ behaves as a function of $m$ by means of series expansion when $x$ gets arbitrarily large.
In order to distinguish this power mean function more clearly from the power mean of an arbitrary collection of values, and to avoid possible confusion with imaginary numbers, we use $k$ for the index instead of $i$.
Firstly, we again consider the Pythagorean means:
\begin{flalign*}
M_1(x)&=AM(x)=\frac{1}{x}\sum \limits_{k=1}^{x} k=\frac{1}{x}\cdot\frac{x(x+1)}{2}=\frac{x}{2}+\frac{1}{2},\\
M_0(x)&=GM(x)=\sqrt[x]{\prod\limits_{k=1}^{x}k}=\sqrt[x]{x!}, \\
M_{-1}(x)&=HM(x)=\frac{x}{\sum \limits_{k=1}^{x} \frac{1}{k}}=\frac{x}{H_x}.
\end{flalign*}
Obviously, $M_1(x)\sim\frac{x}{2}+\frac{1}{2}$. 
For the geometric mean, we use the computer to find the Puiseux series expansion at $x=\infty$: \\
$M_0(x)=\Bigg(e^{-x}x^{x}\Bigg(\sqrt{2\pi}\sqrt{x}+\frac{1}{6}\sqrt{\frac{\pi}{2}}\sqrt{\frac{1}{x}}+\frac{1}{144}\sqrt{\frac{\pi}{2}}\Big(\frac{1}{x}\Big)^{3/2}-\frac{139\sqrt{\frac{\pi}{2}}\big(\frac{1}{x}\big)^{5/2}}{25920}-\frac{571\sqrt{\frac{\pi}{2}}\big(\frac{1}{x}\big)^{7/2}}{1244160}\newline+\frac{103879\sqrt{\frac{\pi}{2}}\big(\frac{1}{x}\big)^{9/2}}{104509440}+O\bigg(\Big(\frac{1}{x}\Big)^{11/2}\bigg)\Bigg)\Bigg)^{1/x}.$ \\
For the sake of simplicity, we remove all terms whose exponent of x is fixed, since the $x$th root of those will quickly eradicate them. Thus we can write the following asymptotic equivalence relation:
\[M_0(x)\sim\Big(e^{-x}x^{x}\Big)^{1/x}=\frac{x}{e}.\]
Trivially, $M_{-1}(x)\sim\frac{x}{H_x}$. 
But the generalized Puiseux series only allows combinations of (fractional) powers \cite{puiseux} and logarithms besides elementary arithmetic, so we write, again assisted by the computer,
\[M_{-1}(x)=\frac{x}{\log{(x)}+\gamma}-\frac{1}{2(\log{(x)}+\gamma)^{2}}+\frac{\log{(x)}+\gamma+2}{12x(\log{(x)}+\gamma)^{3}}+O\Bigg(\bigg(\frac{1}{x}\bigg)^{2}\Bigg).\]
Again we ignore the terms that go to $0$:
\[M_{-1}(x)\sim\frac{x}{\log{(x)}+\gamma},\]
which easily follows from the definition of $\gamma$. \\
We now examine the cases of $M_m(x)$, where $m\notin\{-1,0,1\}$.
Let us first observe what happens when $m=2,$ using the formula \cite{bonacci} for square pyramidal numbers:
\begin{flalign*}
\lim_{x \rightarrow\infty}M_{2}(x)&=\lim_{x \rightarrow\infty}\sqrt{\frac{1}{x}\sum\limits_{k=1}^{x}k^{2}}=\lim_{x \rightarrow\infty}\sqrt{\frac{1}{x}\cdot\frac{x(x+1)(2x+1)}{6}}\\
&=\lim_{x \rightarrow\infty}\sqrt{\frac{(x+1)(2x+1)}{6}}.
\end{flalign*}
The Laurent series expansion at $x=\infty$ gives 
\[M_2(x)=\frac{x}{\sqrt[2]{3}}+\frac{\sqrt[2]{3^{1}}}{4}+O\bigg(\frac{1}{x}\bigg).\]
Proceeding similarly, using the sum of power identities we obtain for large $x$
\begin{flalign*}
M_3(x)&=\frac{x}{\sqrt[3]{4}}+\frac{\sqrt[3]{4^{2}}}{6}+O\bigg(\frac{1}{x}\bigg), \\
M_4(x)&=\frac{x}{\sqrt[4]{5}}+\frac{\sqrt[4]{5^{3}}}{8}+O\bigg(\frac{1}{x}\bigg),
\end{flalign*}
\[\cdots\]
and by continuing the pattern we end up with
\begin{equation}
\label{M_mPos} M_m(x)=\frac{x}{\sqrt[m]{m+1}}+\frac{m+1}{2m\cdot \sqrt[m]{m+1}}+O\bigg(\frac{1}{x}\bigg)\text{ for all }m\geq 2, m\in\mathbb{N}.
\end{equation}
\noindent Obviously, $M_m(x)=\frac{x}{1}+\frac{1}{2}+O\bigg(\frac{1}{x}\bigg)$.
Let us now investigate how $\displaystyle\lim_{x \rightarrow\infty} M_m(x)$ behaves when $m\leq-2$ is a negative integer.
For $m=-2$, we get
\[M_{-2}(x)=\sqrt[-2]{\frac{1}{x}\sum\limits_{k=1}^{x}k^{-2}}=\sqrt[2]{\frac{x}{\sum\limits_{k=1}^{x}\frac{1}{k^{2}}}}=\sqrt[2]{\frac{x}{H_x^{(2)}}}.\]
Puiseux series expansion at $x=\infty$ then yields
\begin{flalign*}
M_{-2}(x)&=\frac{\sqrt{6}\sqrt{x}}{\pi}+O\Bigg(\sqrt{\frac{1}{x}}\Bigg) \\
&=\frac{\sqrt[2]{x}}{\sqrt[2]{\zeta(2)}}+O\Bigg(\sqrt{\frac{1}{x}}\Bigg).
\end{flalign*}
Analogously we obtain
\[M_{-3}(x)=\frac{\sqrt[3]{x}}{\sqrt[3]{\zeta(3)}}+O\Bigg(\sqrt{\frac{1}{x}}\Bigg),\] and
\begin{flalign*}
M_{-4}(x)&=\frac{\sqrt{3}\sqrt[4]{10}\sqrt[4]{x}}{\pi}+O\Bigg({\bigg(\frac{1}{x}\bigg)^{11/4}}\Bigg) \\
&=\frac{\sqrt[4]{x}}{\sqrt[4]{\zeta(4)}}+O\Bigg({\bigg(\frac{1}{x}\bigg)^{11/4}}\Bigg),
\end{flalign*}
and continuing the pattern,
\begin{equation}
	\label{M_mNeg}
	M_m(x)=\frac{\sqrt[-m]{x}}{\sqrt[-m]{\zeta(-m)}}+o(1)\text{ for all } m\leq -2\, ,m\in\mathbb{Z},
\end{equation}
and $\displaystyle\lim_{m\rightarrow-\infty}M_m(x)=1$ for any $x\in\mathbb{N}$.
We see that the equations $\eqref{M_mPos}$ and $\eqref{M_mNeg}$ are very different.
\newpage

\subsection{Approaching the harmonic mean}\label{aphar}

We have found out that the power mean $M_m(x)$ is dominated by powers of $x$ when $m<-1$ and by $x$ when $m>-1$. At $m=-1,$ however, there is a sudden shift in the Puiseux series to $\frac{x}{\log{x}}$. We are thus interested in analyzing how precisely this shift occurs. Taking the limit from just one side would not work very well, since it would just blow off, so we use a trick à la Cauchy principal value: Let us define 
\[F(x, \epsilon)\coloneqq\frac{M_{-1+\epsilon}(x)+M_{-1-\epsilon}(x)}{2}\] for $\epsilon>0$. Then we get \[HM(x)\sim\lim_{\epsilon\rightarrow 0} F(x, \epsilon).\]

We thus need a sequence that goes to 0. Let's take $\frac{1}{n}$ with $n\in\mathbb{N}$ and $n\geq 2.$ Then we obtain
\begin{flalign*}
F\bigg(x, \frac{1}{2}\bigg)&=\frac{1}{2\bigg(\frac{H_x^{(1/2)}}{x}\bigg)^{2}}+\frac{1}{2\bigg(\frac{H_x^{(2/3)}}{x}\bigg)^{2/3}},\\
F\bigg(x, \frac{1}{3}\bigg)&=\frac{1}{2\bigg(\frac{H_x^{(2/3)}}{x}\bigg)^{3/2}}+\frac{1}{2\bigg(\frac{H_x^{(4/3)}}{x}\bigg)^{3/4}},\\
F\bigg(x, \frac{1}{4}\bigg)&=\frac{1}{2\bigg(\frac{H_x^{(3/4)}}{x}\bigg)^{4/3}}+\frac{1}{2\bigg(\frac{H_x^{(5/4)}}{x}\bigg)^{4/5}},\\
F\bigg(x, \frac{1}{5}\bigg)&=\frac{1}{2\bigg(\frac{H_x^{(4/5)}}{x}\bigg)^{5/4}}+\frac{1}{2\bigg(\frac{H_x^{(6/5)}}{x}\bigg)^{5/6}},
\end{flalign*}
where $H_x^{(m)}$ denotes the generalized harmonic number. After applying Puiseux series expansion at $x=\infty$, we obtain
\begin{flalign*}
F\bigg(x, \frac{1}{2}\bigg)&=\frac{x}{\sqrt[1]{8}}+\frac{x^{2/3}}{2\,\zeta\big(\frac{3}{2}\big)^{2/3}}-\frac{x^{1/2}}{8\,\zeta\big(\frac{1}{2}\big)^{-1}}+\frac{x^{1/6}}{\frac{3}{2}\,\zeta\big(\frac{3}{2}\big)^{5/3}}+\frac{3\,\zeta\big(\frac{1}{2}\big)^{2}-2}{32}+O\Bigg(\sqrt{\frac{1}{x}}\,\Bigg),\\
F\bigg(x, \frac{1}{3}\bigg)&=\frac{x}{\sqrt[2]{108}}+\frac{x^{3/4}}{2\,\zeta\big(\frac{4}{3}\big)^{3/4}}-\frac{x^{2/3}}{12\sqrt{3}\,\zeta\big(\frac{2}{3}\big)^{-1}}+\frac{x^{5/12}}{\frac{8}{9}\,\zeta\big(\frac{4}{3}\big)^{7/4}}+O\Big(\sqrt[3]{x}\Big),\\
F\bigg(x, \frac{1}{4}\bigg)&=\frac{x}{\sqrt[3]{2048}}+\frac{x^{4/5}}{2\,\zeta\big(\frac{5}{4}\big)^{4/5}}-\frac{x^{3/4}}{24\cdot2^{2/3}\,\zeta\big(\frac{3}{4}\big)^{-1}}+\frac{x^{11/20}}{\frac{5}{8}\,\zeta\big(\frac{5}{4}\Big)^{9/5}}+O\big(\sqrt{x}\,\Big),\\
F\bigg(x, \frac{1}{5}\bigg)&=\frac{x}{\sqrt[4]{50000}}+\frac{x^{5/6}}{2\,\zeta\big(\frac{6}{5}\big)^{5/6}}-\frac{x^{4/5}}{40\,\sqrt[4]{5}\,\zeta\big(\frac{4}{5}\big)^{-1}}+\frac{x^{19/30}}{\frac{12}{25}\,\zeta\big(\frac{6}{5}\big)^{11/6}}+O\Big(x^{3/5}\Big),
\end{flalign*}
where the coefficients for the first term form an interesting sequence. Assisted by the OEIS \cite{chebypoly}, we can write for positive integers $n\geq 2$

\[F\bigg(x, \frac{1}{n}\bigg)=\frac{x}{2\cdot n^{\frac{n}{n-1}}}+o(x).\]

\newpage

\subsection{Unexpected Bernoulli numbers}\label{brnuli}

After we have let $x$ go to $\infty$, the question naturally arises, what if $x\rightarrow 0$? Of course this does not make a lot of sense, since $M_m(x)$ by definition is the mean of all positive integers from 1 to $x$, and there are no positive integers between 1 and 0. Nevertheless, it turns out that by expanding the formulae that we have obtained beyond their original confines, we are able to delve deeper into the nature of numbers. Let's start with the Pythagorean means. We have established that $AM(x)=\frac{x+1}{2}$. Plugging in 0 for $x$, pretending we are allowed to do that, we obtain $AM(0)=\frac{1}{2}$.
By definition, $GM(x)=\sqrt[x]{x!}=\sqrt[x]{\Gamma(x+1)}$. Taking the limit as $x\rightarrow 0$, we get $e^{-\gamma}$.
\newline Analogously, we obtain $\displaystyle\lim_{x \rightarrow 0} HM(x)=\lim_{x \rightarrow 0}\frac{x}{H_x}=\frac{6}{\pi^{2}}=\frac{1}{\zeta(2)}$. We have thus already found three beautiful numbers which are at the very core of number theory: one half, the inverse of the exponential of $\gamma$, and the reciprocal of $\zeta(2)$. Continuing, we use the definition of the power mean $M_m(x)$ for positive integer values of $m\geq 2$:
\begin{flalign*}
M_2(0)&=\lim_{x \rightarrow 0}\sqrt{\frac{1}{x}\cdot\frac{x(x+1)(2x+1)}{6}}=\frac{1}{\sqrt{6}},\\
M_3(0)&=\lim_{x \rightarrow 0}\sqrt[3]{\frac{1}{x}\cdot\frac{x^{2}(x+1)^{2}}{4}}=0,\\
M_4(0)&=\lim_{x \rightarrow 0}\sqrt[4]{\frac{1}{x}\cdot\frac{x(x+1)(2x+1)(3x^{2}+3x-1)}{30}}=\frac{1}{\sqrt[4]{-30}}.
\end{flalign*}
Proceeding in the same fashion we obtain 0 for $m=2n-1$ and the $2n$th root of a fraction for $m=2n$, which is real for $m=4n-2$ and imaginary for $m=4n$. Generalizing, we realize that these numbers are exactly of the form $\sqrt[m]{B_m}$, i.e.
\begin{equation}
	\label{MeanZeroPos}
	M_m(0)=\sqrt[m]{B_m}\text{ for } m\geq 2, \, m\in\mathbb{Z},
\end{equation}
where $B_m$ denotes the $m$th Bernoulli number \cite{brnuli}.
\newline Note that $B_m=-m\,\zeta(-m+1)$ for integers $m\geq 2.$ Considering negative integer values $m\leq -2$, we get
\begin{flalign*}
M_{-2}(0)&=\lim_{x \rightarrow 0}\sqrt[2]{\frac{x}{H_x^{(2)}}}=\frac{1}{\sqrt[2]{2\,\zeta(3)}},\\
M_{-3}(0)&=\lim_{x \rightarrow 0}\sqrt[3]{\frac{x}{H_x^{(3)}}}=\frac{1}{\sqrt[3]{3\,\zeta(4)}},\\
M_{-4}(0)&=\lim_{x \rightarrow 0}\sqrt[4]{\frac{x}{H_x^{(4)}}}=\frac{1}{\sqrt[4]{4\,\zeta(5)}}.
\end{flalign*}
With the same argument it is easy to show that 
\begin{equation}
	\label{MeanZeroNeg}
	M_{m}(0)=\frac{1}{\sqrt[-m]{-m\,\zeta(-m+1)}}=\sqrt[m]{-m\,\zeta(1-m)}\text{ for }m\leq -2,\, m\in\mathbb{Z}.
\end{equation}
Consider again the pythagorean means. For the arithmetic mean, we write
\[AM(0)=M_1(0)=\frac{1}{2}=\sqrt[1]{-1\cdot\bigg(-\frac{1}{2}\bigg)}.\]
Note that $\sqrt[1]{z}=z \text{ for any } z\in\mathbb{R}$ and $\zeta(0)=-\frac{1}{2}$.
Hence,
\[M_1(0)=\sqrt[1]{(-1)\cdot\zeta(0)}=\sqrt[1]{(-1)\cdot\zeta(1-1)}.\]
Similarly we can derive
\[HM(0)=M_{-1}(0)=\frac{1}{\zeta(2)}=\sqrt[-1]{1\cdot\zeta(2)}=\sqrt[-1]{-(-1)\cdot\zeta(1-(-1))}.\]
Combining our overcomplicated expressions for $M_1(0)$ and $M_{-1}(0)$ with $\eqref{MeanZeroPos}$ and $\eqref{MeanZeroNeg}$, we write
\[M_{m}(0)=\sqrt[m]{-m\,\,\zeta(1-m)}\text{ for }m\neq 0,\, m\in\mathbb{Z}.\]
This leaves us with the case $m=0$, the geometric mean. As usual when dealing with this mean, we take the limit:
\[\lim_{m\rightarrow 0}M_{m}(0)=\lim_{m\rightarrow 0}\sqrt[m]{-m\,\,\zeta(1-m)}.\]
We claim
\begin{theorem}
\label{thm4.2}
$\displaystyle L\coloneqq\lim_{m\rightarrow 0}\sqrt[m]{-m\,\,\zeta(1-m)}=e^{-\gamma}$.
\end{theorem}
\begin{proof}
Consider the functional equation for the Riemann zeta function: \cite{riemn}
\begin{equation}
	\label{FunctionalZeta}
	\zeta(s)=2^{s}\,\pi^{s-1}\, \sin\bigg(\frac{\pi s}{2}\bigg)\,\Gamma(1-s)\,\zeta(1-s).
\end{equation}
First we need to prove the following

\begin{lemma}
\label{lem}
\[\lim_{x\rightarrow 0}\sqrt[x]{-2^{1-x}\,\pi^{-x}\sin\bigg(\frac{\pi\cdot(1-x)}{2}\bigg)\,\zeta(x)}=1.\] 
\end{lemma}
\begin{proof}
Part I. We show that 
\[\Lambda \coloneqq\lim_{x\rightarrow 0}\sqrt[x]{-2\,\zeta(x)}=2\pi.\]
Proof of part I.
\begin{flalign*}
\Lambda&=\lim_{x\rightarrow 0}\big(-2\,\zeta(x)\big)^{\frac{1}{x}} \\
&=\lim_{x\rightarrow 0} \exp{\log{\bigg(\big(-2\,\zeta(x)\big)^{\frac{1}{x}}\bigg)}} \\
&=\exp{\lim_{x\rightarrow 0} \log{\bigg(\big(-2\,\zeta(x)\big)^{\frac{1}{x}}\bigg)}}.
\end{flalign*}
Hence,
\begin{flalign*}
\log{\Lambda}&= \lim_{x\rightarrow 0} \log{\bigg(\big(-2\,\zeta(x)\big)^{\frac{1}{x}}\bigg)} \\
&=\lim_{x\rightarrow 0} \frac{1}{x}\Big(\log{(2)}+\log{(-\zeta(x)})\Big) \\
&=\lim_{x\rightarrow 0} \frac{\log{(2)}+\log{(-\zeta(x))}}{x}.
\end{flalign*}
Since the numerator and denominator are differentiable on $0<\abs{x}\leq 0.5$ and we have a $\frac{0}{0}$ situation, we can apply L'Hôpital's rule.
\begin{flalign*}
\log{\Lambda}&=\lim_{x\rightarrow 0}\frac{\frac{d}{dx}\big(\log{(2)}+\log{(-\zeta(x))}\big)}{\frac{d}{dx}\big(x\big)} \\
&=\lim_{x\rightarrow 0} \frac{\zeta^{'}(x)}{\zeta(x)} \\
&=\lim_{x\rightarrow 0} \frac{-\frac{1}{2}\log{(2\pi)}}{-\frac{1}{2}} \\
&=\log{(2\pi)}.
\end{flalign*}
Therefore,
\[\Lambda=2\pi.\]
\indent Part II. We show that
\[\lambda\coloneqq\lim_{x\rightarrow 0}\sqrt[x]{\sin\bigg(\frac{\pi\cdot(1-x)}{2}\bigg)}=1.\]
Proof of part II.
\begin{flalign*}
\lambda&=\lim_{x\rightarrow 0} \sqrt[x]{\cos{\frac{\pi x}{2}}} \\
&=\exp{\lim_{x\rightarrow 0}\frac{1}{x}\log{\cos{\frac{\pi x}{2}}}}. \\
\end{flalign*}
Taking the logarithm on both sides, we obtain
\[\log{\lambda}=\lim_{x\rightarrow 0}\frac{\log{\cos{\frac{\pi x}{2}}}}{x}.\]
Since the numerator and denominator are differentiable on $0<\abs{x}\leq 0.5$ and we have a $\frac{0}{0}$ situation, we can apply L'Hôpital's rule.
\begin{flalign*}
\log{\lambda}&=\lim_{x\rightarrow 0}\frac{\frac{d}{dx}\big(\log{\cos{\frac{\pi x}{2}}}\big)}{\frac{d}{dx}\big(x\big)} \\
&=\lim_{x\rightarrow 0}\frac{-\frac{1}{2}\pi\tan{\frac{\pi x}{2}}}{1} \\
&=0.
\end{flalign*}
Therefore,
\[\lambda=\exp{0}=1.\]
\indent Since $\Lambda$ and $\lambda$ both converge, we can finish the proof:
\begin{flalign*}
\lim_{x\rightarrow 0}\sqrt[x]{-2^{x-1}\,\pi^{-x}\sin\bigg(\frac{\pi\cdot(1-x)}{2}\bigg)\,\zeta(x)}&=\frac{1}{2\pi}\cdot\Lambda\cdot\lambda \\
&=\frac{1}{2\pi}\cdot 2\pi\cdot 1 \\
&=1.
\end{flalign*}
\end{proof}

\noindent Setting $s=1-m$ in \eqref{FunctionalZeta}, we obtain
\[\zeta(1-m)=2^{1-m}\,\pi^{-m}\, \sin\bigg(\frac{\pi\cdot(1-m)}{2}\bigg)\,\Gamma(m)\,\zeta(m)\]
and multiplying by $-m$ we get
\[-m\,\,\zeta(1-m)=-m\cdot2^{1-m}\,\pi^{-m}\, \sin\bigg(\frac{\pi\cdot(1-m)}{2}\bigg)\,\Gamma(m)\,\zeta(m).\]
Taking the $m$-th root and passing to the limit this gives
\begin{flalign*}
\lim_{m\rightarrow 0}\sqrt[m]{-m\,\,\zeta(1-m)}&=\lim_{m\rightarrow 0}\sqrt[m]{-m\cdot2^{1-m}\,\pi^{-m}\, \sin\bigg(\frac{\pi\cdot(1-m)}{2}\bigg)\,\Gamma(m)\,\zeta(m)} \\
&=\lim_{m\rightarrow 0}\sqrt[m]{\bigg(m\,\,\Gamma(m)\bigg)\cdot\bigg(-2^{1-m}\,\pi^{-m}\, \sin\bigg(\frac{\pi\cdot(1-m)}{2}\bigg)\,\zeta(m)\bigg)}.
\end{flalign*}
\noindent Since $\displaystyle\sqrt[m]{m\,\,\Gamma(m)}$ and $\displaystyle\sqrt[m]{(-2^{1-m}\,\pi^{-m}\, \sin\bigg(\frac{\pi\cdot(1-m)}{2}\bigg)\,\zeta(m)}$ 
both converge as $m\rightarrow 0$, we can transform the limit of the product into a product of limits:
\begin{flalign*}
L&=\lim_{m\rightarrow 0}\sqrt[m]{m\,\,\Gamma(m)} \,\cdot \,\lim_{m\rightarrow 0}\sqrt[m]{-2^{1-m}\,\pi^{-m}\, \sin\bigg(\frac{\pi\cdot(1-m)}{2}\bigg)\,\zeta(m)} \\
&=\lim_{m\rightarrow 0}\sqrt[m]{m\,\,\Gamma(m)} \,\cdot \,\lim_{m\rightarrow 0}\frac{1}{\pi}\sqrt[m]{-2^{1-m}\, \sin\bigg(\frac{\pi\cdot(1-m)}{2}\bigg)\,\zeta(m)}.
\end{flalign*}
As the right limit is equal to one by Lemma \ref{lem}, only the left factor remains. Applying the fundamental property of the gamma function, we finish the proof of \thmref{thm4.2}:
\[L=\lim_{m\rightarrow 0} \sqrt[m]{\Gamma(m+1)}=GM(0)=e^{-\gamma}.\]
\end{proof}
\noindent Thus we have established that
\[M_{m}(0)=\sqrt[m]{-m\,\,\zeta(1-m)}\text{ for all }m\in\mathbb{Z}.\]
\newpage

\subsection{Means on the unit interval}\label{interval}

So far, we have only taken means over collections of numbers, but we can also take means of functions \cite{Comenetz}:

\begin{equation}
\label{amfunction}
\bar f \coloneqq\frac{1}{b-a}\,\int_{a}^{b} f(x)\, dx,
\end{equation}
where $\bar f$ is the arithmetic mean of $f$ on the interval $[a,b]$. \\
\indent Let $f$ be the identity function. Then the analogue to the generalized mean on an interval is given by

\begin{equation}
\label{interval1}
M_m[a,b]\coloneqq\sqrt[m]{\frac{1}{b-a}\,\int_{a}^{b} x^{m}\,dx} \text{ for integers } m\neq 0.
\end{equation}
\indent To obtain the geometric mean, we go back to Definition \ref{amfunction}. We see that $\bar f$ resembles a continuous version of the usual arithmetic mean as discussed at the beginning of Section \ref{pyth}. Let us now start with the usual geometric mean and design a continuous version. We have
\[GM(x_1,...,x_n)=\prod\limits_{i=1}^{n} x_i^{\frac{1}{n}}=\exp{\sum\limits_{i=1}^{n} \log{x_i^{\frac{1}{n}}}}.\]
For the continuous version, we replace the series with a definite integral going from $a$ to $b$ and change the exponent $\frac{1}{n}$ to $\frac{1}{b-a}.$ Then we obtain
\[GM[a,b]\coloneqq\exp{\,\frac{\int_{a}^{b}\log{x}\,dx}{b-a}}.\]
We focus on the unit interval $[a,b]=[0,1]$. This gives
\[M_m[0,1]=\sqrt[m]{\int_{0}^{1} x^{m}\,dx} \text{ for nonzero integers } m,\]
and for $m=0$ \[GM[0,1]=\exp{\int_{0}^{1}\log{x}\,dx}.\]

\noindent Evaluation by means of elementary calculus gives $AM[0,1]=\frac{1}{2}$, $GM[0,1]=\frac{1}{e}$, and $HM[0,1]=0$.
Notice that we have already seen the first two values at the beginning of Section \ref{pyth}, as the coefficients of the leading term in the Laurent expansions. For the harmonic mean, the $0$ can be identified with the coefficient $\frac{1}{\log{(x)}+\gamma}$ in the corresponding Laurent expansion for $M_{-1}(x)$, where x is large.
\newpage
\noindent The general formula \eqref{interval1} can easily be simplified to \[M_m[0,1]=\frac{1}{\sqrt[m]{m+1}} \text{ for any }m\in\mathbb{N},\]
which is exactly the coefficient of the leading term in equation \eqref{M_mPos}. 
Because of this similarity we have to say that there is a striking connection between the reals on the interval $[0,1]$ and the positive integers on the interval $[1,\infty)$, at least concerning generalized means.
\newpage

\section{Pythagorean Series}\label{pythsrs}

We have already discussed the Pythagorean means. However, there are also arithmetic, geometric, and harmonic series, which we will call "Pythagorean series", in resemblance to the corresponding means. Perhaps the most famous one is the geometric series, where it is well known that 
\[\sum\limits_{k=0}^{\infty} x^{k}=\frac{1}{1-x}\text{ for any }\abs{x}<1.\]
The summation over powers of $x$ does look quite similar to the equation for the exponential function,
\[e^{x}=\sum\limits_{k=0}^{\infty}\frac{x^{k}}{k!}\text{ for any }x\in\mathbb{R}.\]
Thus, it seems only natural to insert the factor $\frac{1}{k!}$ also into the arithmetic and harmonic series and see what happens. For the arithmetic series, where the first term is equal to the common difference, we get
\[\sum\limits_{k=0}^{\infty} \frac{x+k}{k!}=e\cdot(x+1).\]
This is easy to prove by splitting the fraction and using the power series of $e^{x}$. Considering the harmonic series, we decide to write them in the generalized form 
\[\sum\limits_{k=1}^{\infty} k^{x}=\zeta(-x),\]
where $\zeta(x)$ is the Riemann zeta function and the special case $x=1$ gives $\emph{the}$ harmonic series $\sum\limits_{k=1}^{\infty}\frac{1}{k}$.
We thus write 
\[\sum\limits_{k=1}^{\infty} \frac{k^{x}}{k!}=\colon h(x).\]
We notice that for positive integers $n$, $h(n)=c_n e,$ with $c_n$ being some positive integer with the first few values $1,2,5,15,52,203,877$ for $n=1,2,3,4,5,6,7$, respectively. We observe \cite{bell} that these are exactly the same values as the corresponding Bell numbers $B_n$, which is exactly what Dobínski's formula \cite{dobinski} states. We note that $h(0)=e-1$ (this means in some sense that the 0th Bell number $B_0=1-\frac{1}{e}$), and $h(-1)=\operatorname{Ei}(1)-\gamma$, where $\operatorname{Ei(x)}$ denotes the exponential integral $\int_{-\infty}^{x}\frac{e^{t}}{t}\,dt$. Taking the limits to the infinities, it is easy to see that
\[\lim_{x\rightarrow\infty}h_x = \infty\text{ and }\lim_{x\rightarrow-\infty}h_x= 0.\]
For negative integers $\leq-2$, the closed form for the values of $h(n)$ involve generalized hypergeometric functions \cite{HQF}, which lie beyond the scope of this paper. Just to give an example, $h(-2)=\,_3F_3(1,1,1;2,2,2;1)\approx 1.14649907$.

However, there are other ways we can play with the harmonic series that do not necessitate the use of generalized hypergeometric functions. One of them, alternation, is discussed in the next section.

\newpage

\section{Natural Logarithm}\label{logs}

\subsection{Double series formula}\label{dblsrs}
It is well known that the alternating harmonic series \[\sum\limits_{n=1}^{\infty} \frac{(-1)^{n+1}}{n}=1-\frac{1}{2}+\frac{1}{3}-\frac{1}{4}+...\]
converges to $\log{2}$. This raises the question if there are any equivalent series for $\log{3}$, $\log{4}$, and so on. At first glance it seems impossible to generalize the above series, since the only thing that distinguishes it from the original harmonic series are the signs, and there are only two additive signs, + and $-$. However, we can rewrite the series as
\[\sum\limits_{n=1}^{\infty} \frac{1}{2n-1}-\frac{1}{2n}\]
because of the following
\begin{lemma}
\[\sum\limits_{n=1}^{\infty} \frac{1}{2n-1}-\frac{1}{2n}=\log{2}.\]
\end{lemma}
\begin{proof}
The digamma function \cite{digamma} at half-integral values is
\[\psi^{(0)}\bigg(q+\frac{1}{2}\bigg)=-\gamma -2\log{2}+2\sum\limits_{k=1}^{q} \frac{1}{2k-1}.\]
Rearranging gives
\[\sum\limits_{k=1}^{q} \frac{1}{2k-1}=\frac{\psi^{(0)}\big(q+\frac{1}{2}\big)+\gamma}{2}+\log{2}.\]
Furthermore, we have
\[\sum\limits_{k=1}^{q} \frac{1}{2k}=\frac{H_{q}}{2}=\frac{\psi^{(0)}(q+1)+\gamma}{2}.\]
Hence
\[\sum\limits_{k=1}^{q} \frac{1}{2k-1}-\frac{1}{2k}=\log{2}+\frac{\psi^{(0)}\big(q+\frac{1}{2}\big)-\psi^{(0)}(q+1)
}{2}.\]
Since $\psi^{(0)}$ is concave on $[1,\infty)$, taking the limit as $q\rightarrow\infty$ we obtain
\begin{align*}
\sum\limits_{k=1}^{\infty} \frac{1}{2k-1}-\frac{1}{2k} &=\log{2}+\lim_{q\rightarrow \infty}\frac{\psi^{(0)}\big(q+\frac{1}{2}\big)-\psi^{(0)}(q+1)}{2}\\
&=\log{2}.
\end{align*}
\end{proof}
Now it is clearer that the alternating harmonic series is concerned with the residues of the positive integers upon division by $2$. Obviously, the same number of terms are added and subtracted, to guarantee convergence. We thus hypothesize that \[\log{3}=\sum\limits_{n=1}^{\infty} \frac{1}{3n-2}+\frac{1}{3n-1}-\frac{2}{3n},\] which can be verified computationally. Furthermore,
\[\log{m}=\sum\limits_{n=1}^{\infty}\bigg(-\frac{m-1}{mn}+\sum\limits_{k=1}^{m-1} \frac{1}{mn-k}\bigg)\text{ for all } m\geq 2,\, m\in\mathbb{N}\]
in general and
\[\log{4}=\sum\limits_{n=1}^{\infty} \frac{1}{4n-3}+\frac{1}{4n-2}+\frac{1}{4n-1}-\frac{3}{4n},\] in particular.
These can be rewritten as
\begin{equation}
	\label{power}
	 \log{m}=\sum\limits_{n=1}^{\infty}\sum\limits_{k=1}^{m-1} \bigg(\frac{1}{mn-k}-\frac{1}{mn}\bigg),
\end{equation}
and
\[\log{4}=\sum\limits_{n=1}^{\infty} \bigg(\frac{1}{4n-3}-\frac{1}{4n}\bigg)+\bigg(\frac{1}{4n-2}-\frac{1}{4n}\bigg)+\bigg(\frac{1}{4n-1}-\frac{1}{4n}\bigg),\] 
respectively. We now prove
\begin{theorem}
\label{six}
\[\sum\limits_{n=1}^{\infty}\sum\limits_{k=1}^{m-1} \Bigg(\frac{1}{mn-k}-\frac{1}{mn}\Bigg)=\log{m} \quad\quad
\text{ for all integers } m\geq 2.
\]
\end{theorem}
\begin{proof}
\begin{flalign*}
S_b\coloneqq\sum\limits_{n=1}^{b}\sum\limits_{k=1}^{m-1} \Bigg(\frac{1}{mn-k}-\frac{1}{mn}\Bigg)&=\sum\limits_{n=1}^{b}\sum\limits_{k=0}^{m-1} \Bigg(\frac{1}{mn-k}-\frac{1}{mn}\Bigg) \\
&=\sum\limits_{n=1}^{b}\Bigg(-\frac{1}{n}+\sum\limits_{k=0}^{m-1} \frac{1}{mn-k}\Bigg) \\
&=-H_{b}+\sum\limits_{n=1}^{b}\sum\limits_{k=0}^{m-1} \frac{1}{mn-k} \\ 
&=-H_{b}+\sum\limits_{n=1}^{b}\sum\limits_{k=0}^{m-1} \frac{1}{(n-1)m+m-k}. 
\end{flalign*}
Substituting $j\coloneqq m-k$, we write
\begin{flalign*}
S_b&=-H_{b}+\sum\limits_{n=1}^{b}\sum\limits_{j=1}^{m} \frac{1}{(n-1)m+j}\\
&=-H_{b}+\sum\limits_{n=1}^{b}\big(H_{n\cdot m}-H_{(n-1)\cdot m}\big)\\
&=-H_{b}+H_{m\cdot b}.
\end{flalign*}
Taking the limit as $b \rightarrow\infty$ and applying the definition of the Euler-Mascheroni constant $\gamma$ we end up with:
\begin{flalign*}
\lim_{b \rightarrow\infty}\sum\limits_{n=1}^{b}\sum\limits_{k=1}^{m-1} \bigg(\frac{1}{mn-k}-\frac{1}{mn}\bigg)&=\lim_{b \rightarrow\infty} H_{mb}-H_{b}\\
&=\lim_{b \rightarrow\infty}\log{(mb)}\,+\gamma-(\log{(b)}\,+\gamma)\\
&=\lim_{b \rightarrow\infty}\log{(mb)}\,-\log{b}\\
&=\log{m}  \quad\quad \text{ for all integers } m\geq 2.
\end{flalign*}
\end{proof}

We shall proceed with the use of $\log{4}$ as an example.\\
\newpage

\subsection{First power logarithm}

Inspired by the similarity with the Riemann zeta function for positive integers $s\geq 2$, $\zeta(s)=\sum\limits_{n=1}^{\infty} \frac{1}{n^{s}}$, we also bring in powers. Two sensible ways for generalizing \thmref{six} are either by raising every term to a specific power, in our example this would be
\[\log^{(2), 1}{4}\coloneqq\sum\limits_{n=1}^{\infty} \bigg(\frac{1}{4n-3}\bigg)^{2}-\bigg(\frac{1}{4n}\bigg)^{2}+\bigg(\frac{1}{4n-2}\bigg)^{2}-\bigg(\frac{1}{4n}\bigg)^{2}+\bigg(\frac{1}{4n-1}\bigg)^{2}-\bigg(\frac{1}{4n}\bigg)^{2},\] 
or by first subtracting the two terms in brackets as shown in \eqref{power}, and then raising to said power, e.g.
\[\log^{(2), 2}{4}\coloneqq\sum\limits_{n=1}^{\infty} \bigg(\frac{1}{4n-3}-\frac{1}{4n}\bigg)^{2}+\bigg(\frac{1}{4n-2}-\frac{1}{4n}\bigg)^{2}+\bigg(\frac{1}{4n-1}-\frac{1}{4n}\bigg)^{2}.\] Here the superscript $^{(2)}$ denotes the power of 2. Generalizing $\log_1^{(2)}{4}$ for positive integers $m\geq 2$ and powers $p\in\mathbb{R}$, we get
\[\log^{(p), 2}{m}\coloneqq\sum\limits_{n=1}^{\infty} \sum\limits_{k=1}^{m-1} \Bigg(\bigg(\frac{1}{mn-k}\bigg)^{p}-\bigg(\frac{1}{mn}\bigg)^{p}\Bigg).\] 
We have thus created a vast array of possible combinations of values for $m$ and $p$, the most worthy of consideration seem to be the ones, where one quantity is fixed and the other approaches infinity. Let us first observe the case $p=2$, with varying $m$:
\[\log^{(2), 1}{m}=\sum\limits_{n=1}^{\infty} \sum\limits_{k=1}^{m-1} \Bigg(\bigg(\frac{1}{mn-k}\bigg)^{2}-\bigg(\frac{1}{mn}\bigg)^{2}\Bigg).\] 
\noindent For $m=2$, we have
\begin{flalign*}
\log^{(2), 1}{2}&=\lim_{b\rightarrow\infty}\sum\limits_{n=1}^{b} \sum\limits_{k=1}^{2-1} \Bigg(\bigg(\frac{1}{2n-k}\bigg)^{2}-\bigg(\frac{1}{2n}\bigg)^{2}\Bigg) \\
&=\lim_{b\rightarrow\infty} \frac{1}{12}\bigg(-3\,\psi^{(1)}\bigg(b+\frac{1}{2}\bigg)+3\,\psi^{(1)}\big(b+1\big)+\pi^{2}\bigg)\\
&=\frac{\pi^{2}}{12}+\frac{1}{4}\lim_{b\rightarrow\infty}\bigg(\psi^{(1)}\big(b+1\big)-\psi^{(1)}\bigg(b+\frac{1}{2}\bigg)\bigg),
\end{flalign*}
where $\psi^{(1)}$ is the trigamma function, which approaches zero as its argument gets large. Since the limit goes to $0-0=0$, we get
\[\log^{(2), 1}{2}=\frac{\pi^{2}}{12}.
\]

\noindent Similarly, we can compute more values to arrive at the following list $\Big(m,\log^{(2), 1}{m}\Big)$:
\[\bigg(2,\frac{\pi^{2}}{12}\bigg),\bigg(3,\frac{\pi^{2}}{9}\bigg),\bigg(4,\frac{\pi^{2}}{8}\bigg),\bigg(5,\frac{2\pi^{2}}{15}\bigg),\bigg(6,\frac{5\pi^{2}}{36}\bigg),\bigg(7,\frac{\pi^{2}}{7}\bigg),\bigg(8,\frac{7\pi^{2}}{48}\bigg),\bigg(9,\frac{4\pi^{2}}{27}\bigg),\]

\noindent which are all of the form $\displaystyle\bigg(m,\frac{m-1}{6m}\cdot\pi^{2}\bigg)$, where the second component obviously approaches $\displaystyle\frac{\pi^{2}}{6}=\zeta(2)$ as $m$ gets arbitrarily large. 
For $p=3$, we proceed similarly to obtain $\displaystyle\Big(m,\log^{(3), 1}{m}\Big)$: \\
\[\bigg(2,\frac{3\zeta(3)}{4}\bigg)\,, \bigg(3,\frac{8\zeta(3)}{9}\bigg)\, ,\bigg(4,\frac{15\zeta(3)}{16}\bigg)\, ,\bigg(5,\frac{24\zeta(3)}{25}\bigg).\] \\

\indent Now we fix $m=2$ and vary $p$; we compute $\Big(p,\log^{(p), 1}{2}\Big)$: \\
\[\bigg(2,\frac{\pi^{2}}{12}\bigg)\,,\bigg(3,\frac{3\,\zeta(3)}{4}\bigg)\,,\bigg(4,\frac{7\,\pi^{4}}{720}\bigg)\,,\bigg(5,\frac{15\,\zeta(5)}{16}\bigg)\,,\bigg(6,\frac{31\,\pi^{6}}{30240}\bigg).\]
We can rewrite this list as \\
\[\bigg(2,\frac{1}{2}\,\zeta(2)\bigg)\,,
\bigg(3,\frac{3}{4}\,\zeta(3)\bigg)\,,
\bigg(4,\frac{7}{8}\,\zeta(4)\bigg)\,,
\bigg(5,\frac{15}{16}\,\zeta(5)\bigg)\,,
\bigg(6,\frac{31}{32}\,\zeta(6)\bigg).\]

\noindent Since $\displaystyle\lim_{s\rightarrow\infty} \zeta(s)=1$, it trivially follows that 
\[\lim_{p \rightarrow\infty}\log^{(p), 1}{2}=1.\]
Fixing $m=3$ and varying $p$ we compute $\Big(p,\log^{(p), 1}{3}\Big)$: \\
\[\displaystyle\bigg(2,\frac{2}{3}\zeta(2)\bigg)\,,\bigg(3,\frac{8}{9}\zeta(3)\bigg)\,,\bigg(4,\frac{26}{27}\zeta(4)\bigg)\bigg(5,\frac{80}{81}\zeta(5)\bigg)\,,\bigg(6,\frac{242}{243}\zeta(6)\bigg).\]
This leads us to assume that $\displaystyle\log^{(p), 1}{3}=\frac{3^{p-1}-1}{3^{p-1}}\cdot\zeta(p),$ which makes sense taking into account what we constructed above, namely \[\displaystyle\log^{(2), 1}{m}=\frac{m-1}{6m}\cdot\pi^{2}=\frac{m-1}{m}\cdot\zeta(2) \,\,\,\,\&\,\,\,\, \displaystyle\log^{(3), 1}{m}=\frac{m^{2}-1}{m^2}\cdot\zeta(3).\]
Based on the previous results we have the following
\begin{conjecture}
\[\log^{(p), 1}{m}=\sum\limits_{n=1}^{\infty} \sum\limits_{k=1}^{m-1} \Bigg(\bigg(\frac{1}{mn-k}\bigg)^{p}-\bigg(\frac{1}{mn}\bigg)^{p}\Bigg)=\frac{m^{p-1}-1}{m^{p-1}}\,\zeta{(p)},
\]
\end{conjecture}
\noindent with the
\begin{corollary} For any integer $p\geq 2,$
\[\lim_{m \rightarrow\infty} \log^{(p), 1}{m}=\lim_{m \rightarrow\infty}\sum\limits_{n=1}^{\infty} \sum\limits_{k=1}^{m-1} \Bigg(\bigg(\frac{1}{mn-k}\bigg)^{p}-\bigg(\frac{1}{mn}\bigg)^{p}\Bigg)=\zeta(p).\] 
\end{corollary}

\newpage

\subsection{Second power logarithm}

The second possibility of generalizing \thmref{six} is
\[\log^{(2), 2}{4}\coloneqq\sum\limits_{n=1}^{\infty} \bigg(\frac{1}{4n-3}-\frac{1}{4n}\bigg)^{2}+\bigg(\frac{1}{4n-2}-\frac{1}{4n}\bigg)^{2}+\bigg(\frac{1}{4n-1}-\frac{1}{4n}\bigg)^{2}.\] Generalizing $\log^{(2), 2}{4}$ for positive integers $m\geq 2$ and powers $p\in\mathbb{R}$, we obtain
\[\log^{(p), 2}{m}\coloneqq\sum\limits_{n=1}^{\infty} \sum\limits_{k=1}^{m-1} \bigg(\frac{1}{mn-k}-\frac{1}{mn}\bigg)^{p}.\] 
Let's fix $m=2$ and observe the behavior of $\log^{(p), 2}{2}$ as $p\geq1$ takes different integer values: \\

\begin{tabular}{l|l}
 p&  $\displaystyle\log^{(p), 2}{2}$   \\
\hline
 1&  $\displaystyle\log{2}$    \\
 2&  $\displaystyle\zeta(2)-2\log{2}$    \\
 3&  $\displaystyle\frac{3}{4}\zeta(3)-3\zeta(2)+6\log{2}$  \\
 4&  $\displaystyle\zeta(4)-3\zeta(3)+10\zeta(2)-20\log{2}$ \\
 5&  $\displaystyle\frac{15}{16}\zeta(5)-5\zeta(4)+\frac{45}{4}\zeta(3)-35\zeta(2)+70\log{2}$   \\
 6&  $\displaystyle\zeta(6)-\frac{45}{8}\zeta(5)+21\zeta(4)-42\zeta(3)+126\zeta(2)-252\log{2}$    \\
 7&  $\displaystyle\frac{63}{64}\zeta(7)-7\zeta(6)+\frac{105}{4}\zeta(5)-84\zeta(4)+\frac{315}{2}\zeta(3)-462\zeta(2)+924\log{2}$    \\
\end{tabular} \\ \\

\noindent Apart from the obvious prominence of the zeta function with descending input and alternating signs, two striking patterns are recognisable. Firstly, the coefficient of $\log{2}$ in the expansion of $\log^{(p), 2}{2}$ is equal to the central entry in the $2(p-1)$th row of Pascal's triangle, or the binomial coefficient $\binom{2(p-1)}{p-1}$, whilst the coefficient of $\zeta(2)$ is exactly half of the former. Regarding binomial coefficients, the third diagonal from the right, i.e. the coefficients of $\zeta(3)$, is equal to $\frac{3}{4}\binom{2p-4}{p-3}$, and the fourth to $\binom{2p-5}{p-4}$.
Secondly, the leading coefficient of the particular expressions, i.e. the coefficient of $\zeta(p)$ for $p\geq2$ and of $\log{2}$ for $p=1$, is equal to $1$ if $p$ is even, and equal to $\frac{2^{p-1}-1}{2^{p-1}}$ if $p$ is odd. \\
For the sake of visibility, we proceed with writing just the coefficients without the signs and multiplying them by $2$ raised to the power of one less than their corresponding diagonal, e.g. $\frac{45}{4}\zeta(3)\mapsto \frac{45}{4}\cdot 2^{3-1}=45$. Then we obtain \\

\begin{tabular}{l|lllllll}
p&	\\
\hline
1& 1    \\
2& 2&2    \\
3& 3&6&6  \\
4& 8&12&20&20  \\
5& 15&40&45&70&70  \\
6& 32&90&168&168&252&252  \\
7& 63&224&420&672&630&924&924  \\
\end{tabular} \\ \\
where the transformed coefficients of $\log{2}$ occupy the rightmost diagonal, and those of $\zeta(n)$ the n-th diagonal from the right. We see that the entries of the second column are equal to r times the corresponding entries of the first column which lie on the same diagonal, where r is the row number. E.g. $224=7\cdot 32$, and 224 is in the 7th row.
Computation strongly suggests that \[\lim_{p \rightarrow\infty} \log^{(p), 2}{2}=\lim_{p \rightarrow\infty}\sum\limits_{n=1}^{\infty} \bigg(\frac{1}{2n-1}-\frac{1}{2n}\bigg)^{p}=0,\] which makes sense since
$\frac{1}{2n-1}-\frac{1}{2n}=\frac{1}{4n^2-2n}\leq\frac{1}{n^2}\leq\frac{1}{n^{2p}}<1$
for all $n$, and a number less than 1 raised to $p$ tends to $0$ as $p$ tends to $\infty$. \\

Fixing $m$ to an integer bigger than 2 yields quite messy results, therefore we now turn to the case of fixed $p$ and variable $m$. Again, if $p$ exceeds 2, the computations become rather ugly, and we restrict ourselves to the case $p=2$:\\
As we already know, $\log^{(2), 2}{2}=\zeta(2)-2\log{2}.$ Moreover, we obtain
\begin{flalign*}
\log^{(2), 2}{3}&=\frac{5}{27}\,\pi^{2}+\frac{1}{6\,\sqrt{3}}\,\pi-\frac{3}{2}\,\log{3},\\
\log^{(2), 2}{4}&=\frac{3}{16}\,\pi^{2}+\frac{1}{6}\,\pi-\frac{5}{2}\log{4},\\
\log^{(2), 2}{5}&=\frac{14}{75}\,\pi^{2}+\frac{\sqrt{85+38\sqrt{5}}}{60}\,\pi+\frac{\sqrt{5}}{24}\log{\Bigg(\frac{3-\sqrt{5}}{2}\Bigg)}-\frac{25}{24}\log{5},\\
\log^{(2), 2}{6}&=\frac{5}{27}\,\pi^{2}+\frac{53}{120\sqrt{3}}\,\pi-\frac{46}{45}\,\log{2}-\frac{39}{40}\,\log{3}.
\end{flalign*}
We note that the expression corresponding to the argument 5 contains a nested radical. This is interesting, as this most prominently happens with $\sin{\frac{2\pi}{5}}$, but this newly established function $\log^{(2), 2}{m}$ does not seem to have any connection to trigonometric functions. \\
\indent When trying to find a pattern, we notice that although all values consist of combinations of $\pi^{2}$, $\pi$, and logarithms, there is no apparent pattern in the coefficients. It feels as though there should be one, for example the coefficients of the logarithms in the entries for 5 and 6 are fractions whose difference between numerator and denominator is 1; the coefficients of $\pi^{2}$ in the entries for 3 and 4 seem to be related to powers of said numbers, and yet, the parameters are very volatile and no obvious pattern is visible; as soon as new values are taken into consideration, the alleged structure dissipates into irksome disharmony. \\
\indent Comparing both versions of the power logarithm we observe that the first one displays much neater values and can be formalized effortlessly, whilst the second version behaves in a much less predictable fashion, but has more alluring patterns and mystery associated with it. That such a simple decision as where to put the parentheses can produce such discrepancies in the structure of the outcome illuminates part of the fascination that mathematics poses.

\newpage

\appendix

\section{Appendix}

\subsection{Acknowledgments}

The author thanks Dr. Richard Bödi for providing Proof \hyperlink{pf2.2}{2.2}, a simpler version \newline of Proof \hyperlink{pf2.1}{2.1}, and for minor corrections.

\subsection{Glossary}

\begin{tabular}{ll}
$H_{n}$ & $n$-th harmonic number \\
$H_n^{(m)}\coloneqq\sum\limits_{k=1}^{n} \frac{1}{k^{m}}$ & Generalized harmonic number \\
$\psi^{(n)}(x)$ & $(n+1)$-th derivative of $\log{\Gamma(x)}$ \\
$\psi^{(0)}(x)$ & Digamma function \\
$\pi(x)$ & Prime counting function, gives the number of primes less than or equal to $x$ \\
PNT & Prime Number Theorem \\
Hölder mean & A generalization of the Pythagorean means \\
Puiseux series & A power series containing fractional exponents and logarithms \cite{puiseux}\cite{Siegel} \\
Cauchy principle value & Method to assign finite values to divergent series \\
\end{tabular} \\ \\

\end{document}